\documentclass{amsart}
\usepackage{amsmath, amssymb, amsthm}
\usepackage{tikz}
\usepackage{graphicx}
\usepackage{hyperref}

\newtheorem{theorem}{Theorem}[section]

\newtheorem{corollary}[theorem]{Corollary}
\newtheorem{definition}[theorem]{Definition}

\newtheorem{proposition}[theorem]{Proposition}

\DeclareMathOperator{\Real}{Re}
\DeclareMathOperator{\Imaginary}{Im}
\DeclareMathOperator{\Area}{Area}
\renewcommand{\Re}{\Real}
\renewcommand{\Im}{\Imaginary}

\title{Discrete minimal surfaces: old and new}
\author[W.Y. Lam]{Wai Yeung Lam}
\address{Department of Mathematics, University of Luxembourg, 4365 Esch-sur-Alzette, Luxembourg}
\email{wyeunglam@gmail.com}

\author[M. Yasumoto]{Masashi Yasumoto}
\address{Graduate School of Technology, Industrial and Social Sciences, Tokushima University, 2-1 Minamijyousanjima-cho, Tokushima, 770-8506, Japan}
\email{yasumoto.masashi@tokushima-u.ac.jp}

\date{\today}
\begin{document}
\begin{abstract}
We survey structure-preserving discretizations of minimal surfaces in Euclidean space. Our focus is on a discretization defined via parallel face offsets of polyhedral surfaces, which naturally leads to a notion of vanishing mean curvature and a corresponding variational characterization. All simply connected discrete minimal surfaces of this type can be constructed from circle patterns via a discrete Weierstrass representation formula. This representation links the space of discrete minimal surfaces to the deformation space of circle patterns, and thereby to classical Teichmüller theory. We also discuss variants of discrete minimal surfaces obtained by modifying the definition of mean curvature, restricting the variational criterion, or replacing circle pattern data with discrete conformal equivalence, Koebe-type circle packings, or quadrilateral meshes with factorized cross ratios. We conclude with open questions on discrete minimal surfaces.
\end{abstract}
\maketitle

\section{Introduction}

Discrete differential geometry (DDG) is the study of structure-preserving discretizations of classical differential geometry. Its central goal is to develop rich discrete theories that mirror the smooth theory and ultimately converge to it in the limit of refinement. This theoretical study leads immediately to various applications in architectural geometry \cite{Pottmann2015} and computer graphics \cite{Crane2020}.

Among the many approaches, two guiding principles shape this field: a discretization should preserve key geometric structures, and it should give rise to a theory with internal coherence and connections to other mathematical areas. While it is often straightforward to construct discretizations that replicate one or two specific features of a smooth theory, it is substantially more challenging to identify discretizations that are mathematically fruitful—that is, those admitting deeper structures and suggesting new perspectives.

Minimal surfaces serve as a natural testing ground for these ideas. As a classical topic in differential geometry, they can be characterized via various perspectives. Their discretizations have been studied extensively and illustrate both the successes and limitations of various approaches within DDG.

As an example, we review a particular structure-preserving discretization of minimal surfaces related to discrete conformal geometry. This approach captures several essential features of the smooth theory, including a variational characterization and a discrete Weierstrass representation via circle patterns. A key insight is to define mean curvature on polyhedral surfaces through the rate of change of area under parallel surface deformations, following Steiner's formula. This leads to a natural definition of discrete minimal surfaces as those with vanishing mean curvature. The resulting discrete minimal surfaces can be constructed from circle patterns together with their infinitesimal deformations. This construction provides a discrete Weierstrass representation formula that, most importantly, establishes a link between the space of discrete minimal surfaces and the deformation space of circle patterns.

Beyond this specific construction, we examine a range of alternative discretizations of minimal surfaces. These include different notions of discrete mean curvature, variational principles, and discrete conformal structures, as well as links to discrete integrable systems. 

 Historically, the integrable-systems approach has been the driving force in the development of discrete surface theory. It enabled the construction not only of discrete minimal surfaces but also of other classical surfaces, including constant-mean-curvature surfaces, constant-Gaussian-curvature surfaces, and isothermic surfaces \cite{Bobenko1996,BP1999,Hertrich-Jeromin2000}. Subsequently, Bobenko, Pottmann, and Wallner \cite{Bobenko2010a} introduced curvature definitions for polyhedral surfaces via Steiner's formula and showed their compatibility with the integrable-systems approach. The present survey reverses this historical order: we begin with curvature-based notions and then discuss integrable systems. For a comprehensive account of the integrable-systems approach to the discrete surface theory, we refer readers to the book of Bobenko and Suris \cite{Bobenko2008}.

Through these various discretizations of minimal surfaces, we highlight a central challenge in DDG: to identify discretizations that not only preserve isolated features of their smooth counterparts but also yield rich theories with deep connections to broader areas of mathematics.

\section{Classical minimal surfaces}

Minimal surfaces are characterized as the critical points of the area functional and naturally arise in physical phenomena such as soap films. Equivalently, they can be defined by the vanishing of their mean curvature. These surfaces are deeply connected to several branches of mathematics, including differential geometry and mathematical physics. Their local theory is well developed and can be elegantly formulated in terms of holomorphic functions in complex analysis.

\subsection{Mean curvature via Steiner's formula} A particularly fruitful approach to understanding mean curvature, with significant implications for discrete surface theory, is based on Steiner’s formula. Suppose $f:U \subset \mathbb{R}^2 \to \mathbb{R}^3$ is a smooth surface with Gauss map $n:U \to \mathbb{S}^2$. For sufficiently small $t$, we consider parallel surfaces $f^t:= f +t n$ which share the same Gauss map $n^t=n$. Steiner's formula states that the area 2-form $\omega_t$ of the parallel surfaces $f^t(u,v)$ satisfies
\[
\omega_t:=\langle f^t_u \times f^t_v, n \rangle du \wedge dv = (1 + 2H t + K t^2) \langle f_u \times f_v, n \rangle du \wedge dv
\]
where $H$ and $K$ are the mean curvature and the Gaussian curvature of $f$. This formula implies that the mean curvature can be interpreted as the rate of change of the area 2-form under parallel surface deformations
\[
H \omega_0 = \frac{1}{2} \left.\frac{d\omega_t}{dt}\right|_{t=0}.
\]
It thus establishes the equivalence between vanishing mean curvature and being a critical point of the area functional with respect to parallel surface deformations.

Furthermore, it is a classical result that minimal surfaces are also critical points 
of the area functional with respect to all compactly supported variations. As we 
will see in Theorem \ref{thm:disvarinormal}, this broader variational property naturally extends to 
our discrete setting.

\subsection{Weierstrass representation}

To relate minimal surfaces to complex analysis, we consider minimal surfaces under conformal parametrization.

A parametrization of a surface $f(u,v)$ is conformal if $\langle f_u, f_v \rangle = 0$ and $\|f_u\|^2 = \|f_v\|^2$, or equivalently the first fundamental form $I$ is a diagonal matrix with equal diagonal entries.

Furthermore, for any smooth immersion of a surface $f$, its first, second, and third fundamental forms denoted by $I, \, I\hspace{-1mm}I , \, I\hspace{-1mm}I\hspace{-1mm}I$ are known to satisfy
\[
I\hspace{-1mm}I\hspace{-1mm}I-2HI\hspace{-1mm}I+KI=0 \, .
\]
When the surface is minimal ($H \equiv 0$), the equation implies that $I\hspace{-1mm}I\hspace{-1mm}I= -KI$. 

Thus for a conformally parametrized minimal surface $f$, we deduce that its Gauss map $n$ is a conformal map to the Riemann sphere. Composing with the stereographic projection $\sigma:\mathbb{S}^2 \to \mathbb{C} \cup \{\infty\}$, we obtain a meromorphic function $g:=\sigma \circ n:U \to \mathbb{C} \cup \{\infty\}$. One may ask whether the converse is true: given a meromorphic function $g$, can we construct a conformal minimal immersion $f$ whose Gauss map is $g$? The answer is positive, and the construction is known as the Weierstrass representation of minimal surfaces.
\begin{proposition}
Every conformal minimal immersion $f$ can be written in the form
\begin{equation}\label{eq:Wrep1}
df=\mathrm{Re} \begin{pmatrix} 1-g^2 \\ \sqrt{-1}(1+g^2) \\ 2g \end{pmatrix} \omega = \mathrm{Re} \begin{pmatrix} 1-g^2 \\ \sqrt{-1}(1+g^2) \\ 2g \end{pmatrix} \frac{q}{dg} \, ,
\end{equation}
where $(g,\omega)$ is a pair of a meromorphic function $g$ and a holomorphic 1-form $\omega=\hat{\omega}dz \ (z=x+\sqrt{-1}y)$ with $g^2\hat{\omega}$ being holomorphic, and $q=\omega\,dg$ is a holomorphic quadratic differential. Furthermore, the Gauss map $n$ of $f$ is given by
\[
n=\frac{1}{1+|g|^2}\begin{pmatrix}
2\mathrm{Re}(g) \\
2\mathrm{Im}(g)\\
-1+|g|^2
\end{pmatrix}
\]
which is precisely the inverse image under the stereographic projection of $g$.
\begin{figure}
  \centering
    \includegraphics[width=0.5\linewidth]{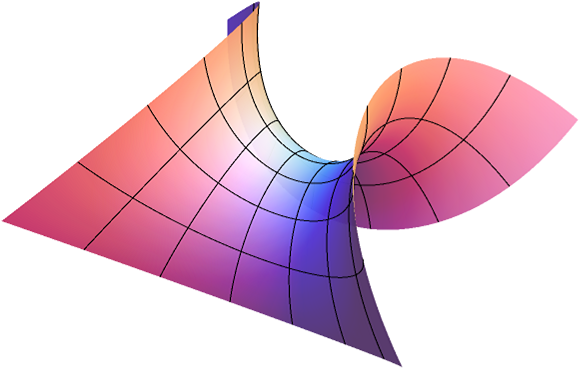}
    \caption{A classical minimal surface in $\mathbb{R}^3$. This  is called Enneper's minimal surface.}
    \label{fig:s_enneper}
\end{figure}
\end{proposition}

\subsection{Further comments on Weierstrass-type representations}

In this paper, we focus only on minimal surfaces in $\mathbb{R}^3$. On the other hand, many other classes of surfaces admit Weierstrass-type representations. For instance, any conformal maximal immersion into Minkowski 3-space can be expressed by a Weierstrass-type representation \cite{Kobayashi1983}. Although the representation formula itself is analogous to the Euclidean case, the global behavior --especially the appearance of singularities on the surface-- is remarkably different. For further information in this direction, see \cite{Fujimori2008,Kobayashi1984,Kokubu2005,Umehara2006}.

Likewise, any surface of constant mean curvature $1$ in hyperbolic $3$-space possesses a Weierstrass-type representation, also known as the Bryant representation; see \cite{Bryant1987,Umehara1993} for details. Analogous representations have also been established for several other classes of surfaces  \cite{Aiyama1999,Galvez2000,Galvez2004,Karl1942}. Recent progress, due to Pember \cite{Pember2020}, shows that all these known Weierstrass-type representations can be unified within a single framework based on transformations of $\Omega$-surfaces.

\section{A discretization of minimal surfaces} \label{sec:discrete}

Given a polyhedral surface in Euclidean space, a central question is in what sense it can be regarded as a discrete analogue of a minimal surface. In this context, it is useful to consider the unit face normals of the polyhedral surface, which serve as a discrete analogue of the Gauss map and are naturally defined on a cell decomposition dual to that of the polyhedral surface.

Let $M = (V, E, F)$ be a cellular decomposition of an oriented surface without boundary, where $V$, $E$, and $F$ denote the sets of vertices, edges, and faces, respectively. Let $M^* = (V^*, E^*, F^*)$ denote its dual decomposition. Assuming no boundary, there is a canonical bijection between the vertices $V$ of $M$ and the dual faces $F^*$ of $M^*$, and similarly between the faces $F$ of $M$ and the vertices $V^*$ of $M^*$.

We define the polyhedral surface on the dual decomposition $M^*$ rather than the primal $M$ so that the faces of the surface correspond directly to the vertices of $M$. Under this convention, the discrete Gauss map, which assigns a unit normal vector to each face, is naturally defined on the vertex set $V$. This structural alignment is essential for establishing a direct link between the normal geometry of the surface and circle patterns on $M$.

Let $f : M^* \to \mathbb{R}^3$ be an oriented polyhedral surface, i.e. a realization of vertices such that the vertices of each face are coplanar. Here we allow polygonal faces to be self-intersecting. For each face $\phi \in F^*$, we define its unit normal vector $n_{\phi} \in \mathbb{S}^2$, giving rise to a Gauss map $n : F^* \to \mathbb{S}^2$. Via the duality $F^* \cong V$, we may alternatively view $n$ as a function $n : V \to \mathbb{S}^2$. For each face $\phi$, we can also measure its signed distance from the origin along the normal direction
\[
h_{\phi}:= \langle f_i, n_{\phi} \rangle
\]
for some vertex $i \in \phi$, which yields a height function $h : V \to \mathbb{R}$. Without loss of generality, we assume that no two adjacent faces are coplanar so that the surface $f$ is completely determined by the pair $(n, h)$. Otherwise we remove the common edge to glue the two coplanar faces together.

For simplicity, we assume that $M^*$ is trivalent, that is,  each vertex is connected to three edges, so that $M$ is triangulated. In cases where $M^*$ has vertices of higher valence, we can interpret it as a trivalent mesh with certain edges degenerated to zero length. This allows for a unified treatment of general polyhedral surfaces within a trivalent setting.

For readers interested in trivalent surface theory, we briefly mention another direction. Kotani, Naito, and Omori developed a discrete surface theory for trivalent graphs in $\mathbb{R}^3$ \cite{Kotani2001,Kotani2017,Kotani2024}, motivated by applications in material science. 

\subsection{Mean curvature}

Given a polyhedral surface $f: M^* \to \mathbb{R}^3$. Let $ij \in E^*$ be an edge connecting vertices $i,j \in V^*$. We denote its length by $\ell_{ij}:=\|f_j-f_i\|$. If the edge length is non-zero, the dihedral angle $\alpha \in (-\pi,\pi)$ is determined via
\[
\sin \alpha_{ij} = \langle n_l \times n_r, \frac{f_j -f_i}{\|f_j-f_i\|} \rangle
\]
where $n_r$, $n_l$ are unit normals of the right face and the left face of the oriented edge from $i$ to $j$. The dihedral angle is independent of the orientation of the edge and thus $\alpha_{ij}=\alpha_{ji}$. 

\begin{definition}\label{def:mean}
  Given a polyhedral surface $f: M^* \to \mathbb{R}^3$, its integrated mean curvature $H: F^* \to \mathbb{R}$ is defined for every face $\phi \in F^*$ by
  \[
  H_\phi = \frac{1}{2} \sum_{ij \in \partial \phi} \ell_{ij} \tan \frac{\alpha_{ij}}{2},\]
  where the sum is over all edges $ij$ within the face $\phi$ involving edge lengths $\ell_{ij}$ and dihedral angles $\alpha_{ij}$ between neighboring faces.

  A polyhedral surface is a discrete minimal surface (also known as a C-minimal surface, see \cite{Lam2015b}) if its integrated mean curvature $H$ vanishes for all faces.
\end{definition}

When the cell decomposition $M^*$ is a quadrilateral mesh such as the $\mathbb{Z}^2$-lattice, the discrete minimal surfaces are reminiscent of classical minimal surfaces parametrized by curvature lines (see Figure \ref{fig:enneper} left). 

\begin{figure}
\centering
\begin{minipage}{0.45\textwidth}
    \centering
    \includegraphics[width=\textwidth]{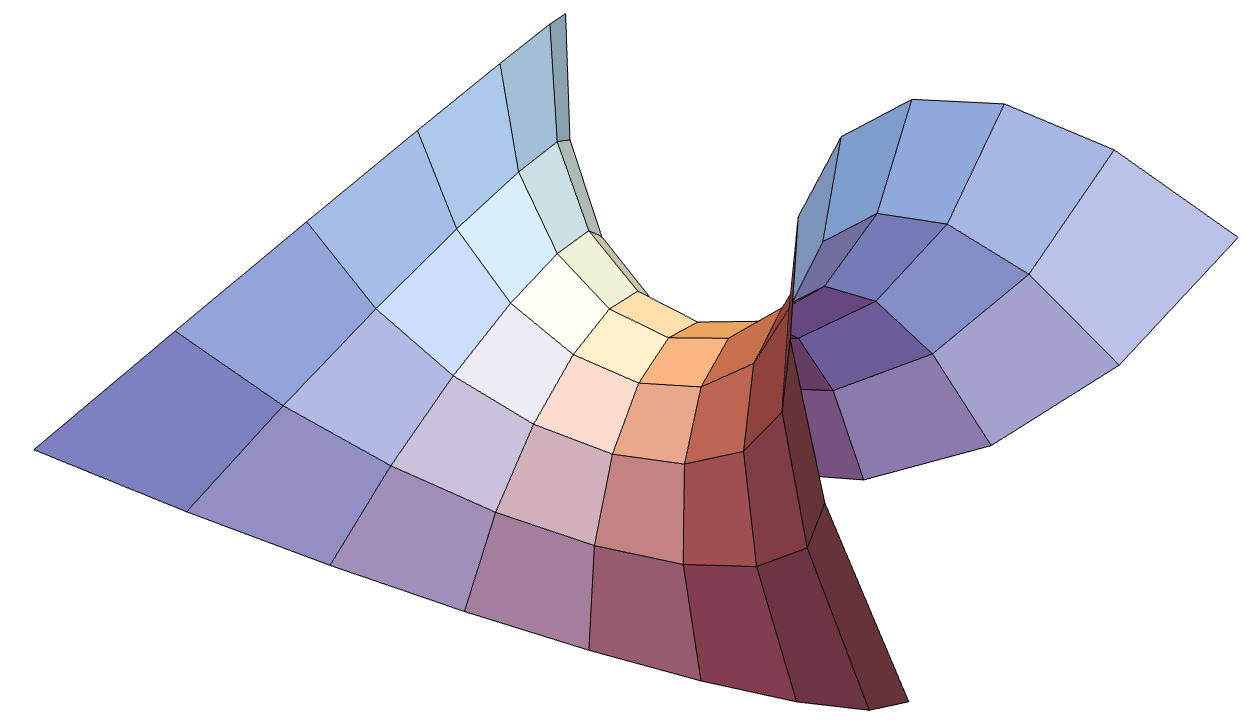}
\end{minipage}
\begin{minipage}{0.45\textwidth}
    \centering
    \includegraphics[width=\textwidth]{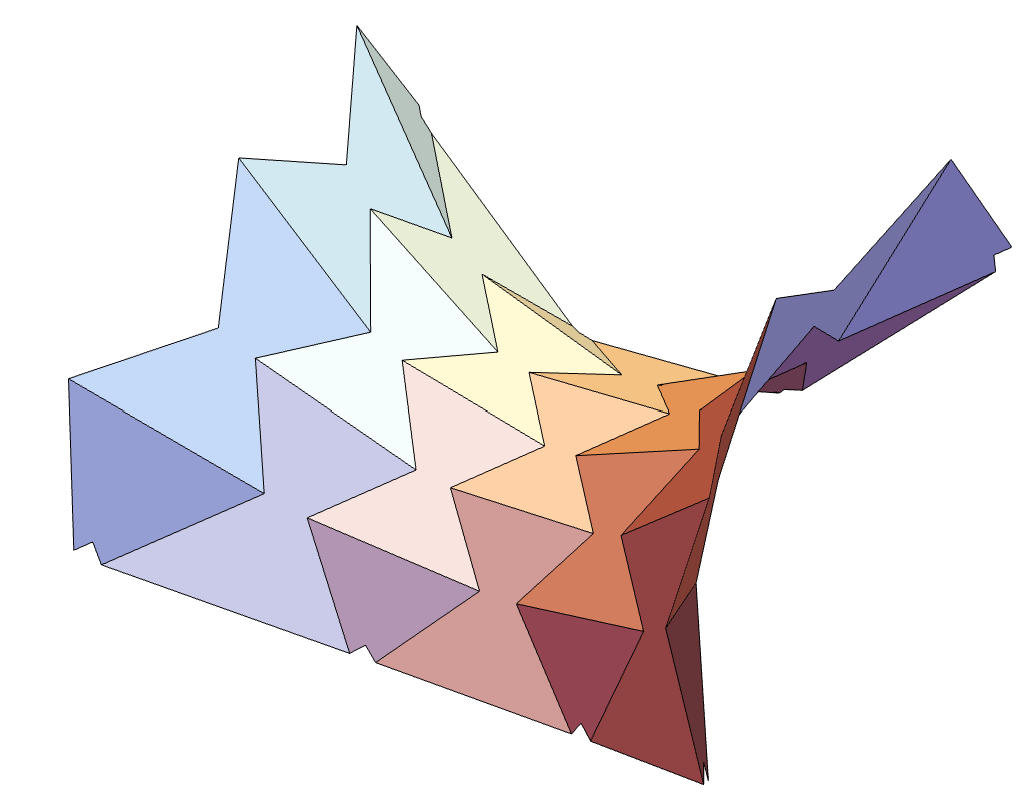}
\end{minipage}
\caption{Two discrete minimal surfaces satisfying Definition \ref{def:mean}, both reminiscent of the classical Enneper minimal surface.}
\label{fig:enneper}
\end{figure}

As an analogue of Steiner's formula, the integrated mean curvature $H_{\phi}$ is related to the change in the signed area of the face
\begin{align}\label{eq:area}
  \Area(f(\phi)) := \frac{1}{2} \sum_{ij \in \partial \phi} \langle f_i \times f_j , n_{\phi} \rangle  
\end{align}
\begin{proposition}\label{prop:area}
	Suppose $f$ is trivalent with non-vanishing edge lengths, determined by its face normal and height function $(n,h)$. Let $\dot{h}$ be an infinitesimal change of the height function and $\dot{f}$ be the corresponding infinitesimal change of vertex position. Then, for every $\phi$, the infinitesimal change in area satisfies
    \begin{equation}\label{eq:infarea}
         \dot{\Area}(f(\phi)) = 2 \dot{h}_{\phi} H_{\phi} + \sum_{ij \in \partial \phi} \frac{\ell_{ij}}{\sin \alpha_{ij}} (\dot{h}_{r(ij)}- \dot{h}_{\phi})
    \end{equation}
    where $ij$ is an oriented edge with left face $\phi$ and the right face $r(ij)$.
\end{proposition}
\begin{proof}
  Differentiating the area formula \eqref{eq:area} gives
    \begin{align*}
        \dot{\Area}(f(\phi)) &= \sum \langle \dot{f}_i \times (f_j-f_i), n_{\phi} \rangle \\
        &=\sum \langle (f_j-f_i) \times n_{\phi} , \dot{f}_i \rangle \\
        &= \sum \langle \ell_{ij}(\frac{n_{\phi}\times n_{r(ij)}}{\sin \alpha_{ij}}) \times n_{\phi} , \dot{f}_i \rangle \\
        &= \sum \frac{\ell_{ij}}{\sin \alpha_{ij}} \langle n_{r(ij)} - n_{\phi} \cos \alpha_{ij} , \dot{f}_i \rangle \\
        &= \sum \frac{\ell_{ij}}{\sin \alpha_{ij}} (\dot{h}_{r(ij)} - \dot{h}_{\phi}\cos \alpha_{ij} ) \\
        &= \sum \frac{\ell_{ij}}{\sin \alpha_{ij}} (\dot{h}_{r(ij)}- \dot{h}_{\phi}) + 2 \dot{h}_{\phi} H_{\phi} 
    \end{align*}
    and the claim follows.
\end{proof}

Setting $\dot{h} \equiv 1$ in Proposition \ref{prop:area} yields a discrete analogue of Steiner's formula for the mean curvature.

\begin{corollary}[Karpenkov-Wallner \cite{Karpenkov2014}]
	Suppose $f$ is trivalent with non-vanishing edge lengths. For $t \in (-\epsilon,\epsilon)$, we offset each face in parallel by the same distance $t$ in the direction of the face normal and obtain a polyhedral surface $f_t$. Then the integrated mean curvature $H$ satisfies, for every face,
	\[
	H_{\phi} = \frac{1}{2} \frac{d}{dt} \Area(f_t(\phi))|_{t=0}.
	\]
\end{corollary}

It is natural to ask whether discrete minimal surfaces can be characterized as critical points of the area functional under infinitesimal deformations. However, this is not the case for arbitrary infinitesimal deformations. We need to consider infinitesimal deformations that preserve the face normals.

\begin{theorem}\label{thm:disvarinormal}
  A trivalent polyhedral surface is a discrete minimal surface if and only if it is a critical point of the area functional with respect to any face offset with finite support.
\end{theorem}
\begin{proof}
 Under any infinitesimal face offset $\dot{h}$ with finite support, Proposition \ref{prop:area} yields that the change in total area is equal to
\begin{align*}
  \sum_{\phi \in F} \dot{\Area}(f(\phi)) = \sum_{\phi \in F^*} 2 \dot{h}_{\phi} H_{\phi}
\end{align*}
and the claim follows.
\end{proof}

\subsection{Weierstrass representation via circle patterns}

 Given a cell-decomposed surface $M=(V,E,F)$, a circle pattern is a realization of vertices $g:V \to \mathbb{C}$ in the plane such that each face has a circumcircle passing through its vertices. If the graph is triangular, every realization is a circle pattern as long as no edge is degenerate to a point (see Figure \ref{fig:circlepattern}). 

\begin{figure}
\centering
\includegraphics[width=0.5\textwidth]{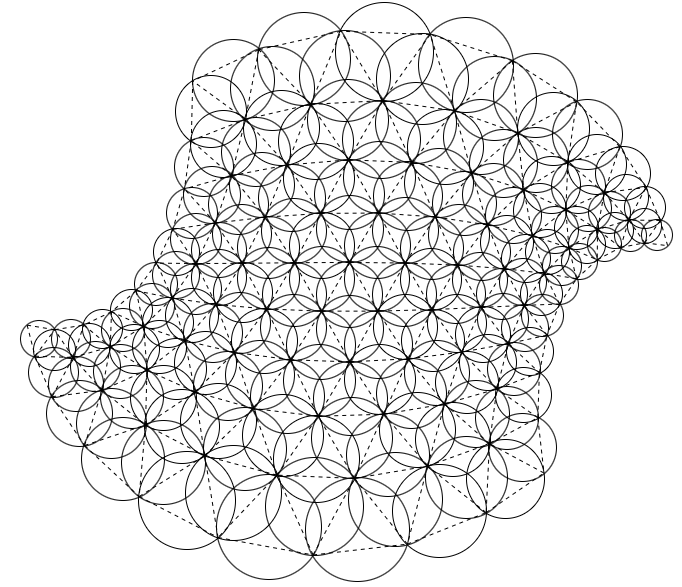}
\caption{A circle pattern with constant intersection angles $\Theta=\pi/3$.}
\label{fig:circlepattern}
\end{figure}

Given a trivalent polyhedral surface in space, composing its Gauss map $n:V \to \mathbb{S}^2$ with the stereographic projection yields a circle pattern in the plane. One could ask whether there is a connection between discrete minimal surfaces and circle patterns analogous to the smooth theory. The answer is positive, and the connection is established via a discrete analogue of the Weierstrass representation for minimal surfaces \cite{Lam2015a}.
\begin{definition}
Let $g: V \rightarrow \mathbb{C}$ be a non-degenerate realization in $\mathbb{C}$. A real-valued function $q: E \to \mathbb{R}$ defined on the set of all edges $E$ is called a {\rm discrete holomorphic quadratic differential} if it satisfies
for all vertices $i \in V$
\begin{equation}\label{eq:hqd}
\begin{split}
&\sum_{j} q_{ij}=0 \, , \\ 
&\sum_{j}\frac{ q_{ij} }{g_j-g_i}=0
\end{split}
\end{equation}
where the sum is taken over all vertices $j$ adjacent to $i$.
\end{definition}

One can verify directly that Equation \eqref{eq:hqd} is equivalent to the condition
\begin{equation}\label{eq:closedc3}
  \sum_{j} \frac{ q_{ij} }{g_j-g_i} \begin{pmatrix} 1-g_i g_j \\ \sqrt{-1}(1+g_ig_j) \\ g_i+g_j \end{pmatrix} =0.
\end{equation}
The term \[ \frac{ q_{ij} }{g_j-g_i} \begin{pmatrix} 1-g_i g_j \\ \sqrt{-1}(1+g_ig_j) \\ g_i+g_j \end{pmatrix} \] is regarded as a discrete 1-form, which assigns to an oriented edge $ij$ a vector in $\mathbb{C}^3$. Reversing the orientation of the edge, i.e. interchanging $i$ and $j$, changes the sign of the vector. 

Thus, Equation \eqref{eq:closedc3} implies that a discrete holomorphic quadratic differential $q$ corresponds to a $\mathbb{C}^3$-valued 1-form that is closed and hence can be integrated, assuming the surface $M$ is simply connected.

\begin{theorem}
[Lam \cite{Lam2015b}] \label{thm:weierstrass}
Given a non-degenerate realization $g: V \rightarrow \mathbb{C}$ and a discrete holomorphic quadratic differential $q: E \rightarrow \mathbb{R}$, there exists a realization $\mathcal{F}: V^* \rightarrow \mathbb{C}^3$ satisfying for every oriented edge from $i$ to $j$
\begin{equation}
\mathcal{F}_l - \mathcal{F}_r=\frac{ q_{ij} }{g_j-g_i} \begin{pmatrix} 1-g_i g_j \\ \sqrt{-1}(1+g_ig_j) \\ g_i+g_j \end{pmatrix}
\end{equation}
where $l$ and $r$ denote the left and the right face of the oriented edge. The imaginary part $f:= \Im(\mathcal{F}): V^* \rightarrow \mathbb{R}^3$ is a discrete minimal surface with Gauss map $n$ given by the inverse stereographic projection of $g$. Conversely, every simply connected discrete minimal surface arises in this way.

\end{theorem}

For every fixed $\theta \in [0,2\pi)$, one can consider $f_{\theta}:= \Im(e^{\theta\sqrt{-1}}\mathcal{F})$ which leads to the associated family of discrete minimal surfaces. In particular, when $\theta=\pi/2$, we obtain the real part $\Re \mathcal{F}=f_{\frac{\pi}{2}}$, which is referred to as an A-minimal surface in \cite{Lam2015b} and reminiscent of classical minimal surfaces parametrized by asymptotic lines. Such discrete minimal surface $f_{\frac{\pi}{2}}$ does not have planar faces but planar vertex stars. The pair $(f_0,f_{\frac{\pi}{2}})= (\Im \mathcal{F}, \Re \mathcal{F})$ is regarded as a discrete analogue of the conjugate pair of minimal surfaces in the smooth theory.

\subsection{Deformation space of circle patterns}

Discrete holomorphic quadratic differentials were introduced in connection with the Weierstrass representation of discrete minimal surfaces. They naturally arise from infinitesimal deformations of circle patterns. The Weierstrass representation thus provides a bridge between the space of discrete minimal surfaces and the deformation space of circle patterns.

For each edge of a circle pattern, we assign the intersection angle between the circumcircles of the two adjacent faces, yielding a function $\Theta: E \to [0, 2\pi)$. These intersection angles serve as a discrete analogue of a conformal structure. Circle patterns sharing the same intersection angles are regarded as having the same discrete conformal structure, and a pair of such circle patterns is interpreted as a discrete conformal map. The angle function is not arbitrary and Rivin \cite{Rivin1996} established necessary and sufficient conditions on $\Theta$ for the existence of circle patterns. For example, one necessary condition is that the sum of intersection angles around each vertex equals $2\pi$.

Up to M\"{o}bius transformations, we consider $P(\Theta)$ the space of circle patterns on the Riemann sphere with prescribed intersection angles $\Theta$, which can be parametrized by complex cross ratios as follows. Given a circle pattern $g:V \to \mathbb{C}\cup\{\infty\}$, the complex cross ratio $X$ associated to an edge $\{ij\}$ shared by triangles $\{ijk\}$ and $\{jil\}$ is defined via (see Figure \ref{fig:orientation}):
\begin{equation}\label{eq:cr}
  X_{ij} :=  -\frac{(g_k - g_i)(g_l -g_j)}{(g_i - g_l)(g_j - g_k)} =X_{ji} \in \mathbb{C}
\end{equation}
It induces a complex valued function on unoriented edges $X:E \to \mathbb{C}$. Two circle patterns induce the same complex cross ratios if and only if they differ by M\"{o}bius transformations. The intersection angles of circles can be read off from the complex cross ratios via the observation
\[
\Theta = \Im \log(X).
\]
One can further ask whether the complex cross ratios satisfy any algebraic relations. Indeed, one finds that for every vertex $i$ with adjacent vertices numbered as $1$, $2$, ..., $r$ in the clockwise order counted from the link of $i$,
	\begin{gather} 
		\Pi_{j=1}^r X_{ij} =1,  \label{eq:crproduct}\\
		X_{i1} + X_{i1} X_{i2} + X_{i1}X_{i2}X_{i3} + \dots +  X_{i1}X_{i2}\dots X_{ir} =0. \label{eq:crsum}
	\end{gather}
Conversely, given a function $X:E \to \mathbb{C}$ satisfying the above algebraic relations, one can construct a circle pattern $g:V \to \mathbb{C}\cup\{\infty\}$ with complex cross ratios $X$ by solving Equation \eqref{eq:cr}. The circle pattern is unique up to M\"{o}bius transformations.

\begin{proposition}
  Up to M\"{o}bius transformations, the space of circle patterns on the Riemann sphere with prescribed intersection angles $\Theta$ can be parametrized by complex cross ratios
  \[
  P(\Theta) :=  \left\{ X: E \to \mathbb{C} \;\middle|\; \Im (\log X) = \Theta \text{ and } X \text{ satisfies } \eqref{eq:crproduct},\, \eqref{eq:crsum} \right\}.
  \]
\end{proposition}

\begin{figure}
\centering
\includegraphics[width=0.5\textwidth]{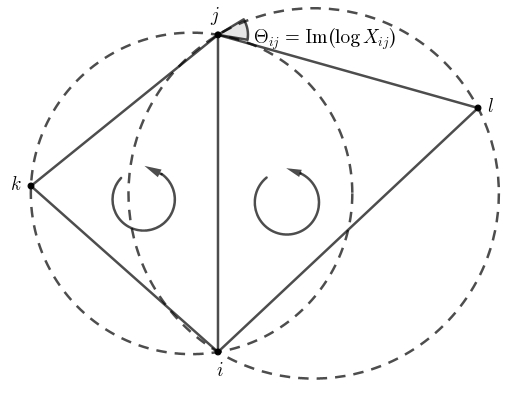}
\caption{A circle pattern with complex cross ratio $X_{ij} = -\frac{(g_k - g_i)(g_l -g_j)}{(g_i - g_l)(g_j - g_k)}$. The intersection angle of the circumcircles of triangles $\{ijk\}$ and $\{jil\}$ is given by $\Theta_{ij} = \arg(X_{ij})$.}
\label{fig:orientation}
\end{figure}

This space is closely related to the complex projective structures of surfaces and hyperbolic geometry. It exhibits rich structures of the Weil-Petersson geometry of the Teichm\"{u}ller space, which is the space of classical conformal structures on surfaces \cite{Lam2019,Lam2024torus,Lam2024pull}.

We further investigate the tangent space of $P(\Theta)$ and explain how infinitesimal deformations of circle patterns correspond to discrete holomorphic quadratic differentials. Suppose $X^{(t)}:E \to \mathbb{C}$ is a 1-parameter family of cross ratios satisfying Equations \eqref{eq:crproduct} and \eqref{eq:crsum} with $X=X^{(t)}|_{t=0}$, then by differentiation the logarithmic derivative $q:= \frac{d}{dt}\left.(\log X^{(t)})\right|_{t=0}$ satisfies for every vertex $i$ with adjacent vertices numbered as $1$, $2$, ..., $r$ in the clockwise order counted from the link of $i$,
 \begin{gather} 
 	\sum_{j=1}^r q_{ij} =0  \label{eq:xcrproduct}\\
 	q_{i1} X_{i1} + (q_{i1}+q_{i2}) X_{i1} X_{i2} + \dots +  (\sum_{j=1}^r q_{ij}) X_{i1}X_{i2}\dots X_{ir} =0 \label{eq:xcrsum}
 \end{gather}
 Furthermore, if $\Im (\log X^{(t)})$ remains constant for all $t$, then $q$ is purely real. By substituting $X_{ij}$ from Equation \eqref{eq:cr} into Equation \eqref{eq:xcrsum}, one can verify that $q$ satisfies Equation \eqref{eq:hqd} and hence is a discrete holomorphic quadratic differential with respect to the realization $g$.

We thus obtain a corollary of the Weierstrass representation formula (Theorem \ref{thm:weierstrass}) that links the space of discrete minimal surfaces to the deformation space of circle patterns.
\begin{corollary} \label{cor:weierstrass}
	Every simply connected discrete minimal surface with face normal $n$ corresponds to an infinitesimal deformation of a circle pattern $g:=\sigma \circ n$ that preserves the intersection angles of the circles, where $\sigma:\mathbb{S}^2 \to \mathbb{C}\cup\{\infty\}$ is the stereographic projection.
\end{corollary}

Besides discrete minimal surfaces in Euclidean space, the space of circle patterns $P(\Theta)$ serves as holomorphic data to construct discrete maximal surfaces in Minkowski space \cite{yashi2017} and discrete constant-mean-curvature-$1$ surfaces in hyperbolic space \cite{Lam2024ams}, which are also defined via integrated mean curvature as in Definition \ref{def:mean}.

Instead of the complex cross ratios, the space of circle patterns can also be parametrized by the Euclidean radii of the circumcircles. In that case, the tangent space is described by discrete harmonic functions with respect to a discrete Laplacian equipped with \emph{cotangent weights} \cite{Lam2015a}.

\subsection{Miscellaneous results}

\subsubsection{Infinitesimal deformations of circumscribed polyhedral surfaces}

The discrete Gauss map \( n : V \to \mathbb{S}^2 \) determines a polyhedral surface circumscribed by the unit sphere. The dihedral angles of this circumscribed surface coincide with the intersection angles of the associated circle pattern. Every infinitesimal deformation of the circle pattern corresponds to an infinitesimal deformation of the polyhedral surface that stays tangent to the sphere and preserves the dihedral angles. Thus, the Weierstrass data can be rephrased in terms of the deformation space of circumscribed polyhedral surfaces with fixed dihedral angles.

It is further shown in \cite{Lam2024Lap} that such infinitesimal deformations are in one-to-one correspondence with infinitesimal isometric deformations of the circumscribed polyhedral surface in \(\mathbb{R}^3\). The corresponding discrete minimal surfaces arise as reciprocal force diagrams, in the sense of the classical Maxwell–Cremona correspondence \cite{Lam2015b,Wallner2008}.

\subsubsection{Discrete spherical Laplacian}

With respect to the Gauss map \( n : V \to \mathbb{S}^2 \), one can define a discrete Laplacian induced by the spherical metric. The area of the corresponding polyhedral surface can be expressed in terms of the height function using this discrete spherical Laplacian; see \cite[Theorem~1.7]{Lam2024Lap}. This formulation is closely related to the variational identity~\eqref{eq:infarea}.

\subsubsection{Delaunay circle patterns and locally convex discrete Gauss maps}

To avoid exotic examples of discrete minimal surfaces, it is natural to impose a local convexity condition on the discrete Gauss map \( n : V \to \mathbb{S}^2 \). Specifically, we require that the Gauss map, viewed as a polyhedral surface, has dihedral angles in the interval \( [0, \pi) \). This condition is equivalent to the associated circle pattern being Delaunay, i.e., the imaginary parts of the logarithmic cross ratios satisfy \( \Im \log X \in [0, \pi) \).

Without this local convexity condition, pathological cases may arise. For instance, there exist discrete minimal surfaces whose underlying topology is that of a sphere (see Figure \ref{fig:jessen}).

\begin{figure}
\centering
\includegraphics[width=0.5\textwidth]{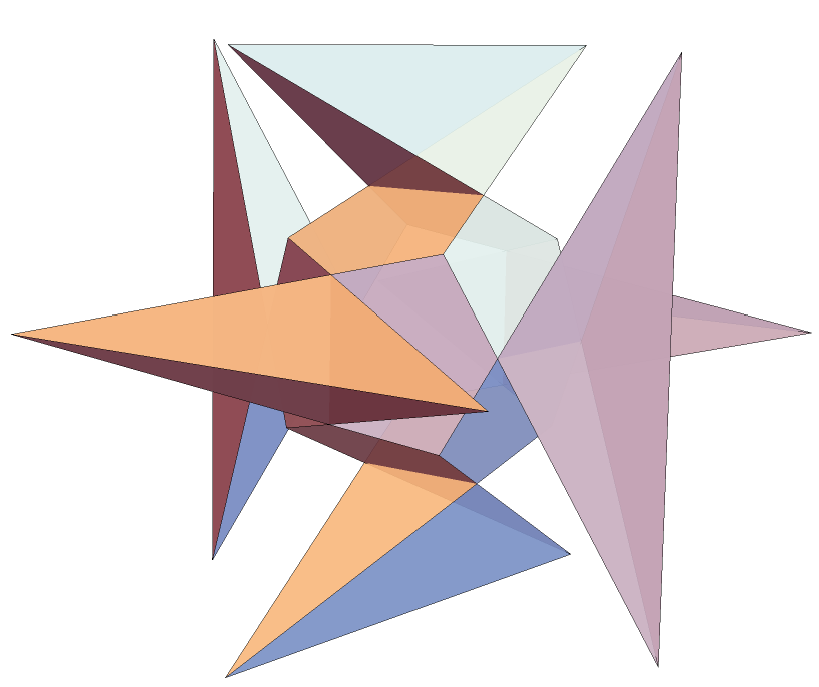}
\caption{A discrete minimal surface which is a realization of a topological sphere and each face has self-intersection. Its Gauss map fails to be locally convex. Indeed, the Gauss map is Jessen's orthogonal icosahedron, which is known to possess a non-trivial infinitesimal isometric deformation.}
\label{fig:jessen}
\end{figure}

\subsubsection{Convergence}

Circle patterns as discrete conformal maps converge to classical conformal maps. Such convergence of Weierstrass data leads to convergence of discrete minimal surfaces \cite{Bobenko2006} and other discretizations of surfaces \cite{Lam2024ams}.

\section{Variants of discretization}

Among the many possible discretizations, a central challenge in discrete differential geometry is to identify those that are mathematically fruitful. Section~\ref{sec:discrete} introduces one such discretization, distinguished by its preservation of several structural properties of classical minimal surfaces. By slightly modifying the discretization, one obtains alternative definitions of discrete minimal surfaces. Some of these variants give rise to a richer theory in certain respects, while others turn out to be more restrictive.

In the following, we explore several such variants and use this opportunity to survey alternative approaches to discrete minimal surfaces. The aim is to illustrate the broader challenge of discrete differential geometry: to discover discretizations that not only capture isolated features of the smooth theory but also support deeper mathematical structures.

\subsection{Variants of mean curvature}

Different definitions of mean curvature naturally give rise to different notions of discrete minimal surfaces.

When applying Steiner's formula in Section \ref{sec:discrete}, we considered parallel face offsets which led to our desired definition of integrated mean curvature. However, one can consider other types of parallel surfaces of a polyhedral surface, such as vertex offset and edge offset. For example, two parallel polyhedral surfaces are said to differ by an edge offset if the distance between the two surfaces is the same along all edges. This approach depends on a choice of vertex normals and has been investigated in \cite{Wallner2008}.

Instead of parallel polyhedral surfaces, one can consider the change in area of the tubular neighborhoods, which leads to another notion of the integrated mean curvature as 
\[
  H(\phi) = \frac{1}{2} \sum_{ij \subset \phi} \ell_{ij} \frac{\alpha_{ij}}{2},
  \]
It also enjoys interesting properties (See \cite[Chapter 28]{PakLectures}). However, its relevance to a broader theory of discrete minimal surfaces is yet to be understood.

\subsection{Restricted variational criterion}

Our discrete minimal surfaces are critical points of the area functional among nearby polyhedral surfaces that have parallel faces. Generally, they might not be critical points of the area functional over all polyhedral surfaces where face normals are allowed to vary.

One can restrict the definition of discrete minimal surfaces by requiring them to be critical points of the area functional over all polyhedral surfaces with fixed combinatorics. This approach is closer to the variational approach by Pinkall and Polthier \cite{Pinkall1993}. It is yet unclear to what extent it admits a Weierstrass representation. Although not all, some discrete minimal surfaces in the sense of Definition \ref{def:mean} satisfy this stronger variational criterion \cite{Lam2015b}.

\subsection{Variants of Weierstrass data/ Gauss maps}

Discrete minimal surfaces can be defined by varying the underlying objects in discrete conformal geometry, beyond circle patterns with fixed intersection angles.

\subsubsection{Discrete conformal equivalence}

Instead of considering circle patterns with fixed intersection angles, one can study circle patterns with fixed length cross ratios, which naturally leads to the notion of discrete conformal equivalence between metrics \cite{Luo2004,Bobenko2010}. This, in turn, gives rise to the concept of conjugate discrete minimal surfaces.

Two assignments of edge lengths $\ell, \tilde{\ell}: E \to \mathbb{R}_{>0}$ are said to be discretely conformally equivalent if there exists a function $u: V \to \mathbb{R}$ such that for every edge $ij \in E$,
\[
\tilde{\ell}_{ij} = e^{\frac{u_i+u_j}{2}}\ell_{ij}.
\]
For two circle patterns in the plane, their Euclidean edge lengths are discretely conformally equivalent if and only if the magnitudes of their complex cross ratios coincide; that is, $|X_{ij}| = |\tilde{X}_{ij}|$ for all $ij \in E$. Analogously to the deformation space $P(\Theta)$ with prescribed intersection angle, one can define the space of circle patterns with prescribed cross ratio magnitudes:
\[
P(\Xi)  = \left\{ X: E \to \mathbb{C} \;\middle|\; \Re(\log X) = \Xi, \text{ and } X \text{ satisfies } \eqref{eq:crproduct},\, \eqref{eq:crsum} \right\},
\]
for a fixed $\Xi: E \to \mathbb{R}$. Similar to Corollary~\ref{cor:weierstrass}, every infinitesimal deformation of a circle pattern with fixed length cross ratio gives rise to another discretization of a minimal surface -- a realization of the dual graph in space with planar vertex stars (instead of planar faces). 
Such surfaces are called A-minimal surfaces in \cite{Lam2015b}, reminiscent of classical minimal surfaces parametrized by asymptotic lines.

Suppose a circle pattern has complex cross ratio $X = e^{\Xi + \sqrt{-1}\Theta}$, so that $X \in P(\Xi) \cap P(\Theta)$. Then the tangent spaces $T_X P(\Xi)$ and $T_X P(\Theta)$ are naturally isomorphic via multiplication by $\sqrt{-1}$ on logarithmic derivatives. Under the discrete Weierstrass representation, this correspondence induces a duality between A-minimal surfaces and C-minimal surfaces as defined in Definition~\ref{def:mean}, mirroring the classical theory of conjugate minimal surfaces.

\subsubsection{Circle Packings}

Instead of intersecting circles, a \textit{circle packing} is a configuration of circles with prescribed tangency patterns \cite{Stephenson2005}. By varying the radii, one obtains many circle packings in the plane with the same tangency pattern. Circle packings were first proposed by Thurston as a discrete analogue of holomorphic functions. This approach also leads to another discretization of minimal surfaces, associated with a different notion of Gauss map.

Through stereographic projection, a planar circle packing can be mapped to a packing on the sphere. Every circle packing on the sphere arises as the intersection of the sphere with a polyhedral surface whose edges are tangent to the sphere. Such a surface is known as a \textit{Koebe polyhedral surface}. Although its vertices lie outside the sphere, it can be interpreted as a discrete Gauss map.

Circle packings thus induce a variant of the discrete Gauss map and mean curvature for polyhedral surfaces in space. Given a polyhedral surface, its discrete Gauss map is taken to be a parallel polyhedral surface of Koebe type, and mean curvature is defined via Steiner's formula with respect to this Gauss map. The vanishing of this mean curvature yields yet another discretization of minimal surfaces, termed \textit{discrete minimal surfaces of Koebe type}~\cite{Ulrike2017}.

These surfaces also admit a discrete Weierstrass representation. Similar to Corollary~\ref{cor:weierstrass}, every infinitesimal deformation of a circle packing with a fixed tangency pattern gives rise to a discrete minimal surface of Koebe type \cite{Lam2018}.

The mean curvature in the setting of circle packings is not well defined for every polyhedral surface, since the discrete Gauss map of Koebe type may not exist. On the other hand, it is shown in \cite{Lam2018} that discrete minimal surfaces of Koebe type can be regarded as special cases of C-minimal surfaces.

\subsection{Integrable systems approach}

In classical differential geometry, minimal surfaces—and more generally, isothermic surfaces—admit rich transformation theories, yielding entire families of related surfaces. The integrable systems approach plays a central role in generating and analyzing these families through algebraic and geometric transformations. Building on the work by Bobenko and Pinkall \cite{Bobenko1996}, we survey a discretization of minimal surfaces via the study of isothermic surfaces \cite{Hertrich-Jeromin2003,NORYPJ}. 

Isothermic surfaces are surfaces that can be parametrized by conformal curvature line coordinates, which are also called \textit{isothermic coordinates}. Examples of isothermic surfaces include surfaces of revolution, minimal surfaces, and surfaces with constant mean curvature. A remarkable feature of isothermic surfaces is their rich transformation theory, which relates one isothermic surface to another through the Christoffel, Goursat, Darboux, and Calapso transformations. Moreover, the class of isothermic surfaces is invariant under Möbius transformations, allowing them to be naturally studied within the framework of Möbius geometry.

One transformation particularly relevant to minimal surfaces is the Christoffel transformation. Suppose $f$ is an isothermic surface in $\mathbb{R}^3$ with curvature line coordinates $(u,v)$. Then there exists a dual surface $f^*$, known as the Christoffel transform of $f$, which can be parametrized by coordinates $(u,v)$ satisfying
\[
f^*_u=\frac{f_u}{\| f_u \|^2} , \ f^*_v=-\frac{f_v}{\| f_v \|^2}.
\]
The surface $f^*$ is isothermic and it is called the Christoffel dual of $f$. The existence of a Christoffel dual also characterizes isothermic surfaces. Furthermore, isothermic surfaces can be characterized by the notion of cross ratios in $\mathbb{R}^3$. For this, identifying the Euclidean space $\mathbb{R}^3$ with the imaginary parts of quaternions, we can define the {\it cross ratio} $Q(A,B,C,D)$ of four points $A,\, B,\, C,\, D \in  \mathbb{R}^3$ as follows:
\[
Q(A,B,C,D):=(A-B)(B-C)^{-1}(C-D)(D-A)^{-1} . 
\]
In general, the quantity $Q(A,B,C,D)$ is quaternion-valued. In particular, like the complex cross ratios of four points on the plane, the cross ratio $Q(A,B,C,D)$ is a nonzero real value if and only if four points $A,\, B,\, C,\, D$ lie on a circle or a line. 

Building upon this setting, an isothermic surface can be also characterized as follows:
\begin{proposition}\label{prop:isothermicity}
Let $f$ be a surface parametrized by curvature line coordinates $(u,v)$. Then $f$ is isothermic if and only if there exist positive real-valued functions $a(u)$ and $b(v)$ such that the following holds for every $(u,v)$:
\begin{equation}\label{eq:infincr}
\lim_{\epsilon \to 0} Q(f(u,v),f(u+\epsilon,v),f(u+\epsilon,v+\epsilon),f(u,v+\epsilon))=-\frac{a(u)}{b(v)}.
\end{equation}
\end{proposition}
This proposition implies that, by re-parametrization, one can choose new coordinates $(\tilde{u},\tilde{v})=(\tilde{u}(u),\tilde{v}(v))$ so that $(\tilde{u},\tilde{v})$ are isothermic coordinates and the cross ratio of each infinitesimal quadrilateral bounded by curvature lines in Equation \eqref{eq:infincr} is $-1$, which is the cross ratio of a square. This is also a modern interpretation of Cayley's description: Cayley \cite{Cayley} introduced isothermic surfaces as those that can be ``divided'' into infinitesimal squares by their curvature lines.

\begin{figure}[ht]
\centering
\includegraphics[width=0.65\textwidth]{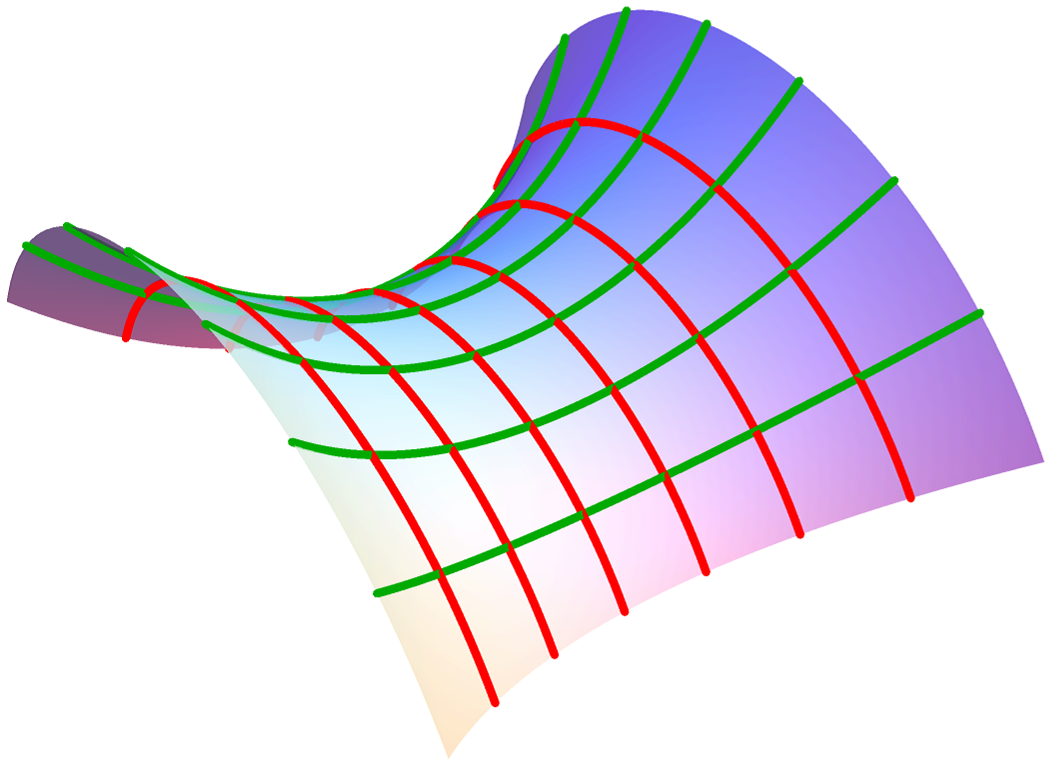}
\caption{An isothermic surface and its curvature lines. Each infinitesimal quadrilateral bounded by curvature lines becomes an infinitesimal square.}
\label{fig:divisible}
\end{figure}

With the Christoffel transformation, a minimal surface can be characterized without referring to mean curvature.

\begin{proposition}
  A surface $f$ is minimal if and only if it is isothermic and its Christoffel transform $f^*$ coincides with its Gauss map to the sphere.
\end{proposition}

Following this approach to minimal surfaces, we recall discrete isothermic surfaces as introduced by Bobenko and Pinkall \cite{Bobenko1996}. Consider a map $f$ from a sublattice $D$ of $\mathbb{Z}^2$ to $\mathbb{R}^3$. For $(m,n) \in D \subset \mathbb{Z}^2$, write
\begin{eqnarray*}
f_i=f(m,n), \ f_j=f(m+1,n), \ f_k=f(m+1,n+1), \, f_{\ell}=f(m,n+1) \, .
\end{eqnarray*}
Mimicking the smooth case, discrete isothermic surfaces in $\mathbb{R}^3$ are defined as follows: A discrete surface $f$ is called {\it discrete isothermic} if it satisfies the following cross ratio condition
\begin{equation}
Q(f_i,f_j,f_k,f_{\ell})=\frac{\alpha_{ij} }{\alpha_{i\ell}}<0 \ \quad \forall (m,n)\in D \, ,
\end{equation}
where $\alpha_{ij}$ (resp. $\alpha_{i\ell}$) is a real-valued function depending only on horizontal direction (resp. vertical direction). Immediate from the definition, the class of discrete isothermic surfaces are preserved under M\"{o}bius transformations as in the smooth theory.

For a given discrete isothermic surface $f$, one can define a discrete version of the Christoffel transform as follows. Let $f$ be a discrete isothermic surface with its cross ratio being $-\alpha_{ij}/\alpha_{i\ell}$. Then there is a discrete surface $f^*$ satisfying
\[
f^*_j-f^*_i=\alpha_{ij}\frac{f_j-f_i}{ \| f_j-f_i \|^2 } \, , \ f^*_{\ell}-f^*_i=\alpha_{i\ell}\frac{f_{\ell}-f_i}{ \| f_{\ell}-f_i \|^2 }
\]
and is called the {\it Christoffel dual} of $f$. Like in the smooth case, one can show that the Christoffel dual is also discrete isothermic.

Following the smooth theory, one defines discrete minimal surfaces using the Christoffel transform, without referring to mean curvature.

\begin{theorem}\cite{Bobenko1996}
 A discrete isothermic surface $f$ is minimal if and only if its Christoffel dual is inscribed in a sphere, i.e. $f^*:D \to \mathbb{S}^2$. Furthermore, a discrete isothermic minimal surface admits a Weierstrass representation in the following form: there exists a discrete isothermic surface $g : D \rightarrow \mathbb{R}^2 \cong \mathbb{C}$ such that

 \begin{equation}
\begin{split}
f_j-f_i=\frac{\alpha_{ij} }{2} \mathrm{Re} \left( \frac{1}{g_j-g_i} \begin{pmatrix} 1-g_ig_j \\ \sqrt{-1}(1+g_ig_j) \\ g_i+g_j \end{pmatrix} \right)  \, , \\
f_{\ell}-f_i=\frac{\alpha_{i\ell} }{2} \mathrm{Re} \left( \frac{1}{g_{\ell}-g_i} \begin{pmatrix} 1-g_ig_{\ell} \\ \sqrt{-1}(1+g_ig_{\ell}) \\ g_i+g_{\ell} \end{pmatrix} \right) \, .
\end{split}
\end{equation}
and the Christoffel dual is given by $f^*=\sigma^{-1} \circ g$, where $\sigma$ is the stereographic projection.

\end{theorem}

\begin{figure}
\centering
\includegraphics[width=0.65\textwidth]{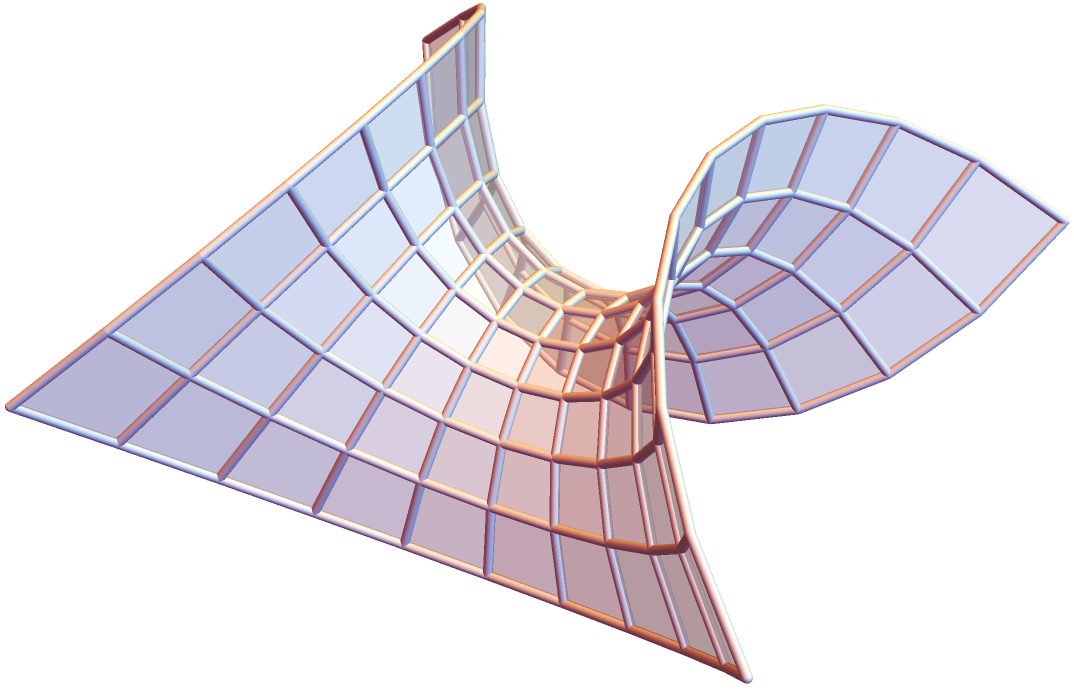}
\caption{A discrete isothermic surface that is minimal. The picture differs from the left picture in Figure \ref{fig:enneper}, although they look similar.}
\label{fig:enneper2}
\end{figure}

The mapping $g$ is a circle pattern since the complex cross ratio of each quadrilateral face is real-valued and hence every face admits a circumcircle. It plays the role of the Weierstrass data for the discrete isothermic minimal surface. Later on, Bobenko, Pottmann, and Wallner \cite{Bobenko2008} introduced curvature definitions for polyhedral surfaces via Steiner's formula and showed their compatibility with discrete isothermic minimal surfaces. In \cite{Lam2015,Lam2015b}, it is further shown that discrete isothermic minimal surfaces are special cases of C-minimal surfaces in Definition \ref{def:mean}. 

Before concluding this section, we emphasize that discrete isothermic surfaces arise naturally within the realm of integrable systems and can be elegantly understood via surface transformations. Applying successive Darboux transformations to an initial smooth isothermic surface generates a sequence of new ones. Because these transformations commute by the Bianchi permutability theorem, the orbit of a fixed point on the initial surface forms a multidimensional lattice in space. A discrete isothermic net then arises precisely as a two-dimensional sub-lattice extracted from these commutative Darboux transformations \cite[Theorem 1.31]{Bobenko1996}.

\subsection{Further comments on recent progress on discrete surface theory}

Recently, in the realm of integrable geometry, discrete surface theories in $\mathbb{R}^3$ have been extended to other space forms. Burstall,  Hertrich-Jeromin, Rossman \cite{Burstall2018} gave a Lie-geometric formulation of discrete linear Weingarten surfaces in space forms. This framework encompasses discrete isothermic surfaces, as well as other discrete surfaces that are not necessarily discrete isothermic; see \cite{NORYPJ} for a detailed exposition. Building on this, Pember, Polly, and the second author \cite{Pember2023} derived Weierstrass-type representations for a variety of discrete surfaces, including several previously known ones obtained in \cite{Hoffmann2012,Rossman2018,Yashi2015}. 

Although this approach enables us to handle a much broader class of discrete surfaces than before, there still remain strong restrictions on the choice of discrete coordinates, and developing a more general discrete surface theory remains an open challenge. To address this issue, following the theory of discrete minimal surfaces presented here, we \cite{yashi2017} described trivalent maximal surfaces in Minkowski $3$-space. While some of the results parallel those for discrete minimal surfaces, the behaviors of the resulting trivalent maximal surfaces are quite different, and this leads to a new research subject to analyze singular behaviors of discrete surfaces. Also, the first author \cite{Lam2024ams} described discrete constant-mean-curvature-$1$ surfaces in hyperbolic $3$-space. Such surfaces can be represented by pairs of circle patterns, and their convergence to the smooth counterparts has been established. It is natural to expect that, analogously to the smooth geometry \cite{Aiyama1999}, one can also develop discrete spacelike constant mean curvature $1$ surfaces in de Sitter $3$-space.

\section{Open Questions}

The study of discrete minimal surfaces remains rich with open questions, both geometric and analytic in nature. We conclude by highlighting several directions that invite further investigation.

\subsection*{Local embeddedness}

The Weierstrass representation via circle patterns (Theorem~\ref{thm:weierstrass}) provides a powerful method for constructing all simply connected discrete minimal surfaces. Nevertheless, even when the underlying circle pattern is Delaunay, it remains unclear whether the resulting surface is locally embedded or may exhibit self-intersections.

\smallskip
\noindent\textit{Question:} Characterize when the discrete minimal surfaces obtained from Delaunay circle patterns have locally embedded faces.

\subsection*{Topology and periodicity}

While the theory of simply connected discrete minimal surfaces is well developed, the construction of examples with non-trivial topology is still in its infancy. In the smooth setting, triply periodic minimal surfaces provide examples of high-genus minimal surfaces and are deeply connected to material science. Developing an analogous discrete theory would be a major step forward.

\smallskip
\noindent\textit{Question:} Develop a systematic framework for constructing discrete minimal surfaces with non-trivial topology, such as triply periodic discrete minimal surfaces.

\subsection*{Integrable structure and general cell decompositions}

Current integrable systems approaches to discrete surface theory rely heavily on quad-mesh structures. Extending these constructions to arbitrary triangulations or general polyhedral surfaces remains an open challenge.

\smallskip
\noindent\textit{Question:} Extend the integrable systems framework for discrete surface theory beyond quadrilateral meshes to general polyhedral cell decompositions.

\subsection*{Global behavior and Bernstein-type results}

In the smooth theory, Bernstein’s theorem asserts that any entire minimal graph in \(\mathbb{R}^3\) must be planar. The discrete analogue of this result is still unknown: can a discrete minimal surface that is a graph over the entire plane be non-planar?

\smallskip
\noindent\textit{Question:} Formulate and prove (or disprove) a Bernstein-type theorem for discrete minimal surfaces.

\section*{Acknowledgements}

This work was partially supported by MEXT Promotion of Distinctive Joint Research Center Program JPMXP0723833165 and Osaka Metropolitan University Strategic Research Promotion Project (Development of International Research Hubs). The second author was also partially supported by JSPS KAKENHI Grant Numbers 25K06978, 23K03091.

\bibliography{survey_dismin}
\bibliographystyle{plain}

\end{document}